\documentclass[12pt]{amsart}
\usepackage{amssymb}
\usepackage[T1]{fontenc}  
\theoremstyle{definition}

\numberwithin{equation}{section}

\theoremstyle{plain}
\newtheorem{thm}{Theorem}[section]
\newtheorem{lmm}[thm]{Lemma}
\newtheorem{cor}[thm]{Corollary}

\theoremstyle{definition}
\newtheorem{dfn}[thm]{Definition}
\newtheorem{rmk}[thm]{Remark}

\def\pmc#1{\setbox0=\hbox{#1}
    \kern-.1em\copy0\kern-\wd0
    \kern.1em\copy0\kern-\wd0}

\def\al{\alpha}
\def\be{\beta}
\def\ga{\gamma}
\def\Ga{\Gamma}
\def\de{\delta}

\def\vep{\varepsilon}

\def\la{\lambda}

\def\si{\sigma}

\def\vp{\varphi}
\def\om{\omega}
\def\op{\operatorname}
\def\ov{\overline}

\def\sm{\setminus}

\def\bs{\boldsymbol}
\begin{document}
\title[Making spaces wild (simply-connected case)]
{Making spaces wild \\ (simply-connected case)}
\author{Katsuya Eda}
\address{School of Science and Engineering, 
Waseda University, Tokyo 169-8555, JAPAN}
\email{eda@waseda.jp}
\begin{abstract}
We attach copies of the circle to points of a countable dense subset $D$ of a separable 
metric space $X$ and
 construct an earring space $E(X,D)$. We show that the fundamental group of $E(X,D)$ is isomorphic to a subgroup of the Hawaiian earring group, if the space $X$ is simply-connected and locally simply-connected. 
In addition if the space $X$ is locally path-connected, the space $X$ can be
 recovered from the fundamental group of $E(X,D)$.   
\end{abstract}

\maketitle
\section{Introduction and definitions}
In former papers \cite{E:spatial, E:embed, ConnerEda:information, E:cut, E:onedim, E:atom} we
investigated fundamental groups of wild spaces, i.e. spaces which are not
semi-locally simply connected. In \cite{ConnerEda:information}, we
defined a construction of spaces from groups and showed that certain
wild spaces, e.g. Sierpinski carpet, Menger sponge and their direct products, are recovered from their fundamental groups (This last
result is extended to a larger class of spaces \cite{E:cut}).

The roots of this investigation are in the following \cite[Theorem
1.3]{E:spatial}: 
\begin{quote}
Let $X$ and $Y$ be one-dimensional, locally path-connected, path-connected, 
metric spaces which are 
not semi-locally simply connected at any point. 
If the fundamental groups of $X$ and $Y$ are isomorphic, 
then $X$ and $Y$ are homeomorphic. 
\end{quote}
In the proof of this theorem we use the fundamental group of 
the Hawaiian earring to extract some topological information from groups. 
 Here, the Hawaiian earring is the plane continuum  
${\mathbb H} = \bigcup _{n=1}^\infty  C_n$ where $C_n = \{ (x,y):
(x+1/n)^2 + y^2 = 1/n^2\}$ and $o$ denotes the origin $(0,0).$  
The fundamental group $\pi _1({\mathbb H})$ is isomorphic to the free 
$\si$-product $\pmc{$\times$}\, \, \;_{n<\om}{\mathbb Z}_n$ \cite{E:free} and is called 
the Hawaiian earring group. 
In the present paper we attach infinitely many copies of the 
circle to a dense subset $D$ of a given simply-connected space $X$ 
and making all the points of $X$ to be non-semilocally simply connected 
in the resulting {\it earring} space $E(X,D)$. Then we'll show that by applying the construction of 
\cite{ConnerEda:information} to the fundamental group of $E(X,D)$ 
we recover the original space $X$. In particular, if $X$ is a separable connected manifold, 
$\pi_1(E(X,D))$ has the same information 
as that of the space $X$, since the choice of $D$ does not affect the space $E(X,D)$. 
The fundamental group of 
$E(X,D)$ is a subgroup of 
$\pmc{$\times$}\, \, \;_{d\in D}{\mathbb Z}_d$. Since $\pmc{$\times$}\, \, \;_{d\in D}{\mathbb Z}_d$ is isomorphic to the Hawaiian earring group, i.e. the fundamental group of the Hawaiian earring. we have many pairwise non-isomorphic subgroups of $\pi _1(\mathbb{H})$ which corresponding to simply-connected separable metric spaces. 
In the next paper \cite{E:earringspace2} 
we treat with non-simply-connected spaces and will have similar results, but in such cases $\pi _1(E(X,D))$ may not be isomorphic to a subgroup of the Hawaiian earring group. This paper only deals with simply-connected spaces $X$. For a non-simply-connected space $X$ the presentation of $\pi _1(E(X,D)$ becomes more complicated, we postpone such cases to the next paper.  
In the remaining part of this section we give preliminary definitions 
and the construction of the space $E(X,D)$. 
Let $X$ be a separable, metric space with its metric $\rho _X$ and 
$D$ be its countable subset of $X$. Since in the present paper $D$ is always dense in $X$, but the construction can be done generally and we shall use it in the next paper. 
We choose 
a point $o$ from the unit circle $\mathbb{S}^1$ and 
denote the metric on $\mathbb{S}^1$ by $\rho _C$, where the diameter of $\mathbb{S}^1$ is 
less than $1$. 
We fix the enumeration $\{ d_n | n\in \mathbb{N}\} = D$. 
We define metrics $\rho _x$ for $x\in X$ on $\{ x\} \times \mathbb{S}^1$ so that $\rho _x((x, u), (x,v)) = \rho _C(u,v)/n$, if $x=d_n$, and on $\{ x\} \times \{ o\}$ so that $\rho _x$ are the constant map $0$,  if 
$x\notin D$. 

Let $E(X,D)$ be a metric space 
space whose base set is 
$X\times\{ o\}\cup \{ d_n | n\in \mathbb{N}\}\times \mathbb{S}^1$ 
with the metric $\rho$ defined by: 
\[
\rho ((x,u),(y,v)) = 
\left\{ 
\begin{array}{ll}
\rho _x(u,v), &\mbox{ if }x=y\in D \\
\rho _X(x,y) + \rho _x(u,o)+ \rho _y(v,o), &\mbox{ otherwise}
\end{array}
\right.
\]
The diameter of a subset $Y$ of $E(X,D)$ is $\sup \{
\rho (x,y) : x,y\in Y\}$ and is denoted by $\op{diam}(Y)$. 
There are two important maps defined on $E(X,D)$. 
Let $r:E(X,D)\to X$ be the retraction defined by $r(x,u) = x$, where 
we identify $x$ and $(x,o)$. Another map is the metric quotient $\si$ obtained by regarding 
$X\times \{ o\}$ as one point $*$, i.e. the distance between $*$ and $(d_n,u)$ is 
$\rho _C(u,o)/n$. The space becomes homeomorphic to the Hawaiian earring. 
  
A path $f$ in a topological space $X$ is a continuous map from a closed
interval $[a,b]$ to $X$. We define $f^-:[a,b]\to X$ by: $f^-(s) =
f(a+b-s)$. 
Two paths $f$ and $g$ in $X$ are simply said to be {\it homotopic}, if
$f$ and $g$ have the same domain $[a,b]$ and are homotopic relative to
the end points $\{ a,b\}$. 
We use a set theoretic notion $\bigcup$ as follows: 
$\bigcup _{i\in I}X_i = \bigcup \{ X_i: i\in I\}$ and 
$\bigcup X = \{ u: u\in x \mbox{ for some }x\in X\}$. 
An ordinal number is the set of all ordinals less than itself. To simply
the notation, we also regard a natural number $m$ as the set $\{
0,1,\cdots ,m-1\}$. The least infinite ordinal is written as $\om$ and
is the set of all natural numbers. 
Since we use the symbol ``$(a,b)$ '' for an ordered pair, we write
$\langle a,b\rangle$ for the open interval between $a$ and $b$. 
We denote the empty sequence by $(\, )$. For a finite sequence $s$ and
an element $x$, the sequence obtained by attaching $x$ to $\si$ is
denoted by $s*(x)$. 
The unit interval $[0,1]$ is denoted by $\mathbb{I}$. 
Undefined notions are standard
\cite{Spanier:algtop}.
\begin{rmk}
There are many ways to make spaces wild. In our way we attach a circle to every point of a dense set. 
As a most explicit method we may attach copies of the Hawaiian earring  instead of circles. In such a 
case we cannot have Corollary~\ref{cor:subgroup}, i.e. $\pi _1(E(X,D))$ is not a subgroup of the Hawaiian earring group for simply-connected spaces. 

The countability of a $\pi _1 (X)$ is crucial for recovering the space $X$ from the fundamental group of the resulting space.  We explain this here. 
Any group can be realized by a connected 
simplicial complex. Let $X$ be such a space 
with its fundamental group isomorphic to that of the Sierpinski carpet and $E$ the resulting space by attaching circles or copies of the Hawaiian earring so that every point of $X$ becomes wild.  
Then, since $X$ is a retract of $E$, our construction produces a space which contains not only 
$X$, but the Sierpinski carpet in a complicated way. 

We also remark a similar construction to $E(X,D)$, i.e. attaching infinitely many circles, was done in \cite{EH:tree} to investigate reflections between group theoretic properties of the fundamental groups and set theoretic ones of uncountable trees.  
\end{rmk}

\section{infinitary words}
To investigate $\pi _1(E(X,D))$, we define letters and words. 
We use a notion of words of infinite length as in
\cite{CannonConner:hawaii1} and \cite{E:free,E:sigma}. 
Here we use this notion following  
\cite{CannonConner:hawaii1} and \cite[p.290]{E:sigma}, which is an
infinitary version of the presentation of free groups using generators,
while the notion in \cite{E:free} is an infinitary version of the
presentation of free products. 

Let $D$ be a set. Then, the letters consist of $a$ and $a^{-}$ for
$a\in D$ and we denote the set $\{ a^-: a\in D\}$ by $D^-$. 
The letter $a^{-}$ is a formal inverse for $a\in D$ and so $D\cap D^- =\emptyset$, but we identify $a^1$ and $(a^{-})^{-}$ with $a$.  

A {\it word} $W$ in $\mathcal{W}(D)$ is a map whose domain $\ov{W}$ is a linearly 
ordered set and $W(\al )\in D\cup D^-$
for $\al\in \ov{W}$ and $\{ \al\in\ov{W}: W(\al ) = a\}$  
is finite for each $a \in D\cup D^-$. 
The inverse word $W^{-}$ is defined as: 
$\ov{W^{-}} = \{ \al ~: \al \in \ov{W}\}$ and $\al ^-\le \be ^-$ if
$\be \le \al$ for $\al ,\be \in \ov{W}$ and $W^{-}(\al ^-) = W(\al )^-$. 

A section $S$ of a linearly ordered set is a subset of the
linearly ordered set such that for $\al ,\ga \in S$ $\al <\be <\ga$
implies $\be \in S$. 
A {\it subword} $V$ of a word $W$ is a restriction of $W$ to a section
of $\ov{W}$, i.e. $\ov{V}$ is a section of $\ov{W}$ and 
$V(\al ) = W(\al )$ for $\al \in \ov{V}$. 
For a section $S$ of $\ov{W}$ the subword obtained by restricting to $S$
is denoted by $W | S$. 
The concatenation of words are defined naturally. (We refer the reader
to \cite{E:free} for the precise definitions.)

A subword $V$ of $W$ is a head of $W$, if  
$V$ is non-empty and $W\equiv VX$ for some $X$, and similarly 
a subword $V$ of $W$ is a tail of $W$, if  
$V$ is non-empty and $W\equiv XV$ for some $X$. 

The equivalence between words in $\mathcal{W}(D)$ is defined as
 follows: 

For $W\in \mathcal{W}(D)$ and $F\Subset D$, 
$W_F$ is a word in $\mathcal{W}(D)$ such that $\ov{W_F} =
\{ \al \in \ov{W}: W(\al ) = a \mbox{ or }a^{-}
\mbox{ for some }a\in F\}$ and $W_F(\al )= W(\al )$. Since $W_F$ is a
 word of finite length, $W_F$ presents an element of the free group
 generated by $F$. Two words $V$ and $W$ in $\mathcal{W}(D)$ are
 equivalent, denoted by $V\sim W$ if $V_F = W_F$ as elements in the free 
 group generated by $F$ for each $F\Subset D$. 
We denote the group consisting of all the equivalence class of words in
$\mathcal{W}(D)$ with the concatenation of words as the multiplication 
by $\pmc{$\times$}\, \, \; ^{\sigma}_{a\in D}\mathbb{Z}_a$, where
$\mathbb{Z}_a$'s are copies of the integer group. For the group
theoretic properties of this group we refer the reader to \cite{E:free,E:atom,
EK:hawaii}. 

Since words $V,W\in \mathcal{W}(D)$ presents elements of a group
$\pmc{$\times$}\, \, \; ^{\sigma}_{a\in D}\mathbb{Z}_a$, we simply write
$V = W$ when $V$ and $W$ are equivalent. 
We write $V\equiv W$, when $V$ and $W$ are the same as words, i.e. there
exists an order-preserving bijection $\vp: \ov{V} \to \ov{W}$ such that
$V(\al ) = W(\vp (\al ))$ for $\al\in \ov{V}$. 
As in the case of words of finite length, reduced words are important. A word $W$ is reduced, 
if any non-empty subword of $W$ is not equivalent to the empty word. 
\begin{lmm}~\label{lmm:reducedword2} 
For a word $W\in \mathcal{W}(D)$, there exists a unique reduced word $W_0\in \mathcal{W}(D)$ 
which is equivalent to $W$. 
\end{lmm}
A proof is an easy corollary to \cite[Theorem 1.4]{E:free} and we omit it. The difference of terminology 
is explained after \cite[Definition 4.3]{E:sigma}.

To define the above equivalence directly and combinatorially Cannon and Conner
\cite[Definition 3.4]{CannonConner:hawaii1} introduced a notion ``complete noncrossing inverse pairing''. 
We use this thinking to construct a homotopy and it is particularly important for the non-simply-connected case \cite{E:earringspace2}.  Since our notation is
somewhat different and we have a different purpose, we define 
only notions ``noncrossing pairing' and ``noncrossing inverse pairing'''.

\begin{dfn}\label{dfn:noncrossing}
A {\it noncrossing pairing} $\Ga$ on a linearly ordered set $S$ is the
 following:
\begin{itemize}
\item[(1)] each element of $\Ga$, called a $\Ga$-{\it pair}, is a pair of elements of $S$, i.e. a
	   subset of $S$ of the cardinality exactly $2$; 
\item[(2)] $x\cap y = \emptyset $ for distinct $\Ga$-pairs $x,y$ and, 
   	   for a $\Ga$-pair $\{ \al, \be\}$, $\{ \ga :\al \le 
	   \ga\le \be\}\subseteq \bigcup \Ga$; 
\item[(3)]  $\be _0 <\al _1$ or $\be _1 < \al _0$ for distinct $\Ga$-pairs $\{ \al _0, \be _0\}, 
	   \{ \al _1,\be _1\}$ satisfying 
	   $\al _0 < \be _0$ and $\al _1<\be _1$, 
	  
\end{itemize}
A {\it noncrossing inverse pairing} $\Ga$ on a word $W\in \mathcal{W}(D)$ 
is a noncrossing pairing on $\ov{W}$ such that $W(\al _0) = W(\al _1)^{-}$ 
for each $\Ga$-pair $\{ \al _0,\al _1\}$. 
Let $\preceq$ be a partial ordering on $\Ga$ defined by: 
$\{ \al _0, \be _0\}\preceq \{ \al _1,\be _1\}$ if $\al _1\le \al _0$
 and $\be _0\le \be _1$. 

Let $\Ga$ be a noncrossing inverse pairing on a word $W$. A subword $V$
 of $W$ is {\it bound} by $\Ga$, if for every $\al \in \ov{V}$ there exists a $\Ga$-pair $\{ \be ,\ga\}$ 
such that $\be \le \al \le \ga$. 
\end{dfn}
The first lemma follows from the definition and the second is
straightforward and we omit their proofs.
\begin{lmm}\label{lmm:bound}
Let $\Ga$ be a noncrossing inverse pairing on a word $W$ and $U$ and $V$
 be subwords of $W$ such that 
\begin{itemize}
\item[(1)] $\al < \be$ for $\al\in \ov{U}$ and $\be\in \ov{V};$
\item[(2)] the subword $W'$ of $W$ restricted to $\{ \al\in \ov{W}: 
	   \be _0\le \al \le \be _1\mbox{ for some }\be _0\in \ov{U},
	   \be _1\in \ov{V}\}$ is bound by $\Ga ;$
\item[(3)] for each $\al\in \ov{U}$ there exist $\al '\in \ov{U}$ and
	   $\be\in \ov{V}$ such that $\al \le \al '$ and $\{ \al ',\be\}$ is 
	   a $\Ga $-pair and similarly 
	   for each $\be\in \ov{V}$ there exist $\be '\in \ov{V}$ and
	   $\al\in \ov{U}$ such that $\be ' \le \be$ and $\{ \al ,\be '\}$ is a 
	   $\Ga.$-pair. 
\end{itemize}
Then, the restriction of $\Ga$ to a word $UV$ is a non-crossing inverse
 pairing. 
\end{lmm}
\begin{lmm}\label{lmm:easy}
Let $W_\mu$ be subwords of a word $W$ for $\mu <\la$
\begin{itemize}
\item[(1)] If $\ov{W_\mu}$ is closed under $\Ga$ and $W_\mu$ is a subword of $W_\nu$ for $\mu \le \nu$, then $\bigcup _{\mu <\la}W_\mu$ is a subword of
	   $W$ and $\bigcup _{\mu <\la}\ov{W_\mu}$ is closed under $\Ga$. 
\item[(2)] If $\ov{W_\mu}$ is closed under $\Ga$ and $W_\nu$ is a subword of $W_\mu$ for $\mu \le \nu$, then $\bigcap _{\mu <\la}W_\mu$ is a subword of
	   $W$ and $\bigcap _{\mu <\la}\ov{W_\mu}$ is closed under $\Ga$. 
\end{itemize}
\end{lmm}
The next lemma is contained in \cite[Section 3]{CannonConner:hawaii1}
implicitly, but we present its proof for the explicitness. 
Though we apply this to words of countable domains, i.e. $\si$-words
\cite{E:free}, we prove this for a general case. 
\begin{lmm}\label{lmm:equivalent}
$($\mbox{Cannon-Conner }\cite{CannonConner:hawaii1}$)$
For a word $W\in \mathcal{W}(D)$, $W = e$, i.e. $W$ is equivalent to the
 empty word, if and only if there exists a noncrossing inverse pairing $\Ga$ 
 on $W$ such that $\bigcup \Ga = \ov{W}$. 
\end{lmm}
\begin{proof}
Let $\Ga$ be a noncrossing pairing on $W$ satisfying $\bigcup \Ga =\ov{W}$, then for each 
$F\Subset D$ the restriction 
$\Ga _F = \Ga |\{ \al \in \ov{W} \, |\, W(\al ) = d \text{ or } d^- \text{ for } d\in F \}$ to $W_F$
 is also a noncrossing inverse pairing $\Ga _F$ on $W_F$ such that 
$\bigcup \Ga _F =  \{ \al \in \ov{W} \, |\, W(\al ) = d \text{ or } d^- \text{ for some} d\in F\}$. Therefore, $W_F = e$ in the free
 product $\ast _{d\in F}\mathbb{Z}_d$ and hence $W = e$. 

To show the converse, let $W = e$. We define a 
noncrossing pairing on $\ov{W}$ inductively as follows. 
We well-order $D$ and so identify $D$ with an ordinal number $\la$. 
We define $\Ga _\mu$ inductively. 
In the $\mu$-step we have defined $\Ga _\nu$ $(\nu < \mu )$ such that 
\begin{itemize}
\item[(1)] each element of $\Ga _\mu$ is a pair of elements of $\ov{W}$; 
\item[(2)] $x\cap y = \emptyset $ for distinct $x,y\in \Ga _\mu$ and 
	   $\bigcup \Ga _\mu = \{ \al\in\ov{W}: W(\al ) \in \mu \mbox{ or }
	   W(\al )^{-} \in \mu \}$;
\item[(3)] for distict pairs $\{ \al _0, \al _1\}, 
	   \{ \be _0,\be _1\}\in \Ga _\mu$ with 
	   $\al _0<\al _1$ and $\be _0<\be _1$, $\al _0 < \be _0$
	   implies $\al _1 <\be _0$ or $\be _1< \al _1$; 
\item[(4)] $W(\al _0) = W(\al _1)^{-}$ for $\{ \al _0,\al _1\}\in \Ga _\mu$; 
\item[(5)] $W | \; \langle\al _0, \al _1\rangle = e$ for each $\{ \al _0, \al _1\} 
	   \in \Ga _\mu$.  
\end{itemize}
We remark that $\Ga _0$ is the empty map by definition. 
When $\mu$ is a limit ordinal, let $\Ga _\mu = 
\bigcup _{\nu <\mu}\Ga _\nu$. Then $\Ga _\mu$ satisfies the required
 properties. 

Let $\mu = \nu +1$. If $\nu$ nor $\nu ^{-}$ does not appear in $W$, we let
 $\Ga _\mu = \Ga _\nu$. Otherwise we extend $\Ga _\nu$ as follows. 

Let $\be _0 < \cdots <\be _m$ be the enumeration of $\{ \al \in \ov{W}:
 W(\al ) =  \nu \mbox{ or } \nu ^{-} \}$. 
For $\be _i$, we have a maximal subword $V$ of $W$ such that 
$\be _i\in \ov{V}$ and $\ov{V} \cap \{ \al_0,\al _1 : \{ \al _0,\al _1\}\in \Ga _\nu , \al _0 < \be _i < \al _1 \} = \emptyset$. 
Then $\ov{V} = \bigcap \{ \langle\al _0,\al _1\rangle : \{ \al _0,\al _1\}\in \Ga _\nu , 
\al _0 < \be _i < \al _1 \}$ and hence $V = e$ by the property (5) of
 $\Ga _\nu$ and Lemma~\ref{lmm:easy}(2). 
We define a word $U$ as the restriction of $V$ to 
\[
 \ov{U} = \{ \be\in\ov{V}: \be\notin [\al _0,\al _1] \mbox{ for any } 
\al _0,\al _1\in \ov{V}\cap \bigcup \Ga _\nu \}. 
\]
Since $U$ is obtained by deleting possibly infinitely many subwords of
 $V$ equivalent to the empty word, we have $U = e$. 
We extend $\Ga _\nu$ on $\{ \be _j: \be _j\in \ov{U}, 0\le j\le m\}$. We
 remark that $\be _i\in \ov{U}$ and we obtain the same $U$ if we start
 from $\be _j\in \ov{U}$. 
We can regard $U$ as a word for the free product $\mathbb{Z}_{\nu}\ast 
\pmc{$\times$}\, \, \; ^{\sigma}_{\xi \ge \mu}\mathbb{Z}_\xi$. 
Since $U = e$, we have a reduction as a word for this free product and
 have a pairing according to the cancellation of $\nu$ and $\nu ^{-}$. (Of
 course the reduction is not unique and so we choose one of them.) 
Now, it is easy to see that the resulting pairing satisfies the required
 properties and we can extend $\Ga _\nu$ on $\{ \al \in \ov{W}:
 W(\al ) =  \nu \mbox{ or }\nu ^{-} \}$ and let $\Ga _\mu$ to be the
 extension. 
We have $\Ga _\la$ as the desired noncrossing inverse pairing. 
\end{proof}
\begin{lmm}~\label{lmm:reducedword1}
Let $W\in \mathcal{W}(D)$, $\Ga$ a maximal noncrossing inverse pairring and 
$\ov{W_0} = \{ \al \in \ov{W} \, |\, \al \notin \bigcup \Ga\}$ and $W_0 = W \, |\, \ov{W_0}$. 
Then, $W_0 = W$ and $W_0$ is reduced. 
\end{lmm}
\begin{proof}
Suppose that $W_0$ is not reduced. Then, there exist $\al < \be$ in $W_0$ such that 
$W_0| \{ \ga \in \ov{W} \, |\, \ga \notin \bigcup \Ga , \al \le \ga \le \be \} = e$. 
By the item (2) in Definition~\ref{dfn:noncrossing}, we can extend $\Ga$ to $\Ga '$ such that 
$\al , \be \in \bigcup \Ga '$, which contradicts the maximality of $\Ga$. 

\end{proof}
For a given countable dense subset $D$ of $X$ we use words in 
$\mathcal{W}(D)$. For each $d\in D = \{ d_n | n\in \mathbb{N}\} $, 
let $C_d$ be the circle attached at $d$ and 
 $w_d$ a letter corresponding to a fixed winding on the circle $C_d$. 

Define $\op{supp}(w_d) = \op{supp}(w_d^{-}) = d$ and for a word $W\in
\mathcal{W}(D)$ let $\op{supp}(W)$ be the set 
$\{ \op{supp}(W(\al )) : \al\in \ov{W}\}$.   

Two paths with domain $[a,b]$ are simply said to be homotopic, if they are homotopic relative 
to $\{ a,b\}$. 
\begin{dfn}
A path $f:[a,b]\to E(X,D)$ is a {\it proper} path, if 
the following hold: 
$f(a),  f(b) \in X$ and if $f(\langle u,v\rangle )\cap X = \emptyset$ and
$f(u)=f(v)\in X$, then $f|[u,v]$ is a winding or reverse-winding on 
the circle $C_d$. 

For a proper path $f$, let $\bigcup _{i<\nu}\langle a_i,b_i\rangle  
= f^{-1}(E(X,D)\sm X)$ where $\nu \le \om$ and $\langle a_i,b_i\rangle
\cap \langle a_j,b_j\rangle 
= \emptyset$ for $i\neq j$. We define $W^f$ to be a word in 
$\mathcal{W}(D)$ such that 
$\ov{W^f} = \{ \langle a_i,b_i\rangle  : i<\nu\}$ with the natural
ordering induced from the one on the real line and 
$W^f (\langle a_i,b_i\rangle ) = w_d$ if
$f|[a_i,b_i]$ is the fixed winding on $C_d$ and 
$W^f (\langle a_i,b_i\rangle ) =
w_d^{-}$ otherwise, i.e. $f|[a_i,b_i]$ is the reverse winding on 
$C_d$. 
\end{dfn}

Recall the metric quotient map $\si$ of $E(X,D)$ to the Hawaiian earring $\mathbb{H}$, where each circle corresponds to $C_d$ for $d\in D$, then $\si _*$ is regarded 
as a homomorphism from $\pi _1(E(X,D))$ to $\pmc{$\times$}\, \, \; _{d\in D}\mathbb{Z}_d$. 
It is easy to see that $\si _*([f]) = [W^f]$ for a proper loop $f$ in $E(X,D)$ with base point $x_0\in X$. 
A word $W\in \mathcal{W}(D)$ is {\it realizable}, if there exists a proper path $f$ in $E(X,D)$ such that $W\equiv W^f$. 

\begin{lmm}\label{lmm:proper}
Let $f:[a,b]\to E(X,D)$ be a path with $f(a), f(b) \in X$.  
Then there exists a proper path homotopic to $f$.
\end{lmm}
\begin{proof}
Let $f^{-1}(E(X,D)\sm X) = \bigcup _{i<\nu }\langle a_i,b_i\rangle $, where
 $\langle a_i,b_i\rangle \cap \langle a_j,b_j\rangle  =\emptyset$ for
 distinct $i,j$ and $\nu \le \om$.  
Then $f|[a_i,b_i]$ is homotopic to the constant, a winding, or a reverse
 winding on the circle at $d\in D$. 
In case $f|[a_i,b_i]$ is homotopic to the constant, we construct a
 homotopy $H$ so that the image of $H\, |\, [a_i,b_i]\times \mathbb{I}$ is in
 some $C_d$ for each $i$. Then we have a proper path
 homotopic to $f$.
\end{proof}

\begin{dfn}\label{dfn:Gamma-bound} Let $f:[a,b]\to E(X,D)$ be a proper path satisfying $f(a), f(b)\in X$ and $W^f = e$  and 
$\Ga$ be a non-crossing inverse pairing on $W^f$. An open subinterval $\langle u, v\rangle$ of 
$[a,b]$ is $\Ga$-{\it bound}, if for any $\vep >0$ there exists  a $\Ga$-pair $\{ \langle a_0, b_0\rangle , \langle a_1, b_1\rangle\}$ such that $u\le a_0 <u+\vep$ and $v-\vep < b_1\le v$. 
A maximal $\Ga$-bound is maximal among $\Ga$-bounds under inclusion. 
A subinterval $[c,d]$ of $[a,b]$ is {\it closed} under $\Ga$, if $f(c)=f(d)\in X$ and 
$W^{f|[c,d]}$ is closed under $\Ga$. 
\end{dfn}

\section{simply-connected case}

Our construction of homotopy is involved in techniques of wild topology. To clarify the construction 
we introduce a definition involved with punctured homotopies and other notions. 
A partial continuous map $H:[a,b]\times [c,d]\to X$ is {\it bound for constant}, if 
$\op{dom}(H)$ contains $\partial ([a,b]\times [c,d])$ and 
$H(s,c)= H(a,t)=H(b,t)$ for $a\le s\le b$ and $c\le t\le d $. A punctured homotopy $H$ is from a loop $f$ to constant, if  $f(s) = H(s,d)$ for $a\le s\le b$.

A {\it punctured homotopy} $H$ is a partial continuous map bound for constant 
satisfying: 
\begin{itemize}
\item[(1)] the domain of $H$ is of form $[a,b]\times [c,d] \sm \bigsqcup _{i\in I}R_i$ where $R_i$ are 
open subrectangles of $[a,b]\times [c,d]$ and $R_i$ are pairwise disjoint; and 
\item[(2)] the restriction of $H$ to each $\partial R_i$ is bound for constant. 
\end{itemize}
We remark that the domain of a punctured homotopy $H$ contains $\partial ([a,b]\times [c,d])$, when we write $H:[a,b]\times [c,d]\to X$, but $\op{dom}(H)$ may not contain the whole $[a,b]\times [c,d]$. 
The {\it size} of a map is the diameter of the range of the map. 
 
A simplest example of a punctured homotopy  is a partial homotopy $H$ bound for constant such that 
$\op{dom}(H) =\partial ([a,b]\times [c,d])$. 

We use the following agreement for the equality of partial maps, which is common in set theory: 
for a partial map $f$ and a set $A$, $\op{dom}(f|A) = \op{dom}(f)\cap A$ and $f|A(x) =f(x)$ for 
$x\in \op{dom}(f|A)$. 

A family $\mathcal{F}$ of loops in a metric space $X$ is {\it admissible}, 
if the following conditions are satisfied. 
\begin{itemize}
\item[(1)]
For any $f\in \mathcal{F}$ and $\de >0$ there exists a punctured homotopy $H$ from $f$ to 
constant such that $\op{diam}(H(\partial R_i)) < \de$ and $H|\partial R_i$ is a punctured homotopy from some $g\in \mathcal{F}$ to constant; 
\item[(2)] For each $\vep >0$ there exists $\de >0$ such that: for 
any $f\in \mathcal{F}$ whose size is less than $\de >0$ and any $\al >0$ there exists a punctured homotopy $H$ from $f$ bound for constant 
whose size is less than $\vep$, $\op{diam}(H(\partial R_i)) < \al$
and $H|\partial R_i$ is a punctured homotopy from some $g\in \mathcal{F}$ to constant for any $R_i$ 
related rectangle to $H$. 
\end{itemize}
\begin{lmm}\label{lmm:main1}
If there exists a complete metric subspace of a space $X$ containing the ranges of all loops in a given  admissible family, 
then every loop in the admissible family is null-homotopic. 
\end{lmm}
\begin{proof} 
First we outline the proof roughly, then proceed to state a rigorous proof. 
Using (2), we have $\de _n>0$ such that any $g\in \mathcal{F}$ whose size is less than $\de _n$ for any $\al$ there exists a punctured homotopy $G$ from $g$ to constant 
whose size is less than $1/n$ such that $\op{diam}(G(\partial R_i)) < \al$
and $G|\partial R_i$ is a punctured homotopy from some $g\in \mathcal{F}$ to constant. 
Let $f\in \mathcal{F}$. Using (1), we have a punctured homotopy $H$ from $f$ to constant 
such that $G|\partial R_i$ is a punctured homotopy from some $g\in \mathcal{F}$ to constant and the size of $G|\partial R_i$ is less than $\de _1$. 
After the $n$-step we have constructed punctured homotopies $H$ so that the sizes of $H|\partial R_i$ are less than $\de _n$ for each $R_i$ related to $H$. In the $n\! +\! 1$-stage we work 
for each $H|\partial R_i$ and construct a punctured homotopy $G$ extending $H|\partial R_i$ 
so that the sizes of  $G$ is less than $1/n$ and also 
the sizes of the the restriction of $G$ to the boundary of each related rectangle is less than 
$\de _{n+1}$. 
The union of all punctured homotopies, we have a continuous map from a dense subset of 
$[0,1]\times [0,1]$ to $X$. Undefined parts are points or intervals which are the intersections of rectangles of intermediate stages. Then, the ranges of punctured homotopies converge by the given condition of the ranges of loops in $\mathcal{F}$. Therefore, we have a homotopy from $f$ to constant, 
i.e. $f$ is null-homotopic. 

Now we start a rigorous proof. Let $Seq$ be the set of all finite sequences $(R_0,\cdots ,R_n)$ of open subrectangles of $[0,1]\times [0,1]$ such that $\ov{R_{i+1}}\subseteq R_i$. 
We define subsets $S_n$ of $Seq$ whose elements are of their length $n$ and punctured homotopies $H_s$ for $s\in S_n$, where 
$(\, )$ is the empty sequence. For $s=((R_0,\cdots ,R_n)\in Seq$, let $s_* = R_n$. 
Let 
$\op{dom}(H_{(\, )}) = \partial ([0,1]\times [0,1])$ and $H_{(\, )}$ the punctured homotopy from a given $f\in \mathcal{F}$ to constant. 
For the $n\! +\! 1$-stage we have defined $S_n$ and $H_s$ for $s\in S_n$. Let $R$ be an open rectangle related to the punctured homotopy $H_s$. 
We recall that the restriction of $H_s$ to $\partial R$ is a punctured homotopy bound for constant
whose size is less than $\de _n$. By (2) we have a punctured homotopy bound for constant of size less than $1/n$ extending $H_s|\partial R$. By adjusting the domain to a subset of $R$ we have a punctured homotopy $H_{s*(R)}$ so that the size of the restriction of $H_{s*(R)}$ to the boundary of 
each open rectangle is less than $\de _{n+1}$. Let $S_{n+1}$ be the set of all such $s\! *\! (R)$ for 
$s\in S_n$. 

After all $n$-stages are over we may have an element $x$ in $[0,1]\times [0,1]$ on which any $H_s$ is not defined. This happens only when there exist infinite sequence $s$ through $\bigcup _n S_n$, i.e. 
$s = (R_0, \cdots , R_n, \cdots )$ and $\ov{R_n}\subseteq R_{n+1}$ and $x\in R_n$ for every $n$. 
Since the sizes of the restriction of punctured homotopies to $\partial R_n$ converges to $0$ and 
the condition of an admissible class, we have a convergent point of the ranges of the maps to which 
we extend the union of punctured homotopies by mapping $x$ to the point. Now we have a homotopy from $f$ to constant. 
\end{proof}

In comparison with proofs of \cite[Theorem
 A.1]{E:free}, \cite[Theorem 3.2]{E:sigma}, \cite[Appendix
 B]{EK:injection} and \cite[Theorem 1.1]{E:edge}, 
the difficulty in our
 case lies in controlling a homotopy in the simply connected space $X$
 and here a noncrossing inverse pairing works.

For a proper loop $f:\mathbb{I}\to E(X,D)$ with $f(0)= f(1) = x$, we have a
 word $W^f$. To show that this correspondence induces an injective
 homomorphism from the fundamental group $\pi _1(E(X,D),x)$ to 
$\pmc{$\times$}\, \, \; ^{\sigma}_{d\in D}\mathbb{Z}_d$, 
it suffices to show that $W^f = e$ if and only if $f$ is
 null-homotopic. 

\begin{thm}\label{thm:main1}
Let $X$ be a simply connected, locally simply connected metric space and
 $x\in X$ and $D$ a countable dense subset of $X$. Then, the metric quotient map $\si$ induces 
an injective homomorphism $\si _*$. 
Consequently, $\pi _1(E(X,D),x)$ is isomorphic to a subgroup of 
$\pmc{$\times$}\, \, \; ^{\sigma}_{d\in D}\mathbb{Z}_d$, i.e. isomorphic to a subgroup of the Hawaiian earring group. 
\end{thm}
\begin{proof}
It suffices to show that $\si _*$ is injective. 
Suppose that $\si _*(f_0)=e$ for a proper loop $f_0$. Then, it implies $W^{f_0} = e$. 
Now it suffices to show that $W^{f_0} = e$ implies that $f$ is null-homotopic. 
We define an admissible class containing $f_0$ the ranges of whose elements are subset of the range of 
$f_0$. Since the range of $f_0$ is compact, the assumptions of Lemma~\ref{lmm:main1} is satisfied and 
we conclude that $f_0$ is null-homotopic. 
Since $W^{f_0} = e$, we have a noncrossing inverse pairing $\Ga$ on $\ov{W^{f_0}}$ 
by Lemma~\ref{lmm:equivalent}. 
Let $\mathcal{F}$ be the set of loops which are obtained from $f_0$ by iterating use of the following 
construction of $g$ 
from a given $f$ such that $f(s) = f(t)\in D^*$ and $W^{f|[s,t]}$ is closed under $\Ga$ let 
\begin{itemize}
\item[(1)] $g = f|[s,t]$; or
\item[(2)] $\op{dom}(g) = \op{dom}(f)$ and $g(u) = f(s)=f(t)$ for $u\in [s, t]$ and 
$g(u) = f(u)$ otherwise.  
\end{itemize}
In the either cases $W^g$ is closed under $\Ga$. We show that $\mathcal{F}$ is admissible, 
which implies that $f_0$ is null-homotopic by Lemma~\ref{lmm:main1}. 
 
To show the first property for $\mathcal{F}$ to be admissible there are two procedures which we iterate alternatively finitely many times. 
For a given $\de >0$ choose $\de _0>0$ so that $|u-v|<\de _0$ implies $\rho (f_0(u),f_0(v))<\de /2$. 
We remark that $|u-v|<\de _0$ implies $\rho (f(u),f(v))<\de /2$ for every $f\in \mathcal{F}$. 

Our induction is on $m$ such that the size of $f$ is less than $m\de _0$. 
Let $m=2$. Since $f(a)=f(b)$, the size of $f$ is less that $\de $. Let 
$H(s,0) = H(a,t) = H(b,t) = f(a)=f(b)$ and $H(s,1) = f(s)$. Then $H:\partial ([a,b]\times [0,1]) \to E(X,D)$ 
is a desired punctured homotopy. 

Our induction hypothesis is that, if the size of $f$ is less than $m\de _0$ we have 
a puctured homotopy $H$ from $f$ to constant such that the sizes of $H(\partial R_i)$ are less than 
$\de$, where $R_i$ are the related rectangles. 

The first procedure is taking the maximal pairwise disjoint set of all maximal $\Ga$-bounds 
$\{ \langle u_i,v_i\rangle \, | \, i\in I\}$. 
Since the range of $r\circ f$ is contained in $X$,  $r\circ f$ is null-homotopic. 
We let $H(s,1/2) = r (f(s))$ and $H| [a,b]\times [0, 1/2] $ to be a homotopy to the constant $f(a) = f(b)$. 
For $t\in [1/2, 1]$ and $s\notin \bigcup _{i\in I}\langle u_i,v_i\rangle $, let 
$H(s, t) = f(s)$ and for $s\in \langle u_i, v_i\rangle$,  
$H(s,1) =f(s)$, but $H(s,t)$ is not defined for $t\in [1/2, 1]$. 
Consequently, $H\, |\, \partial ([u_i,v_i]\times [1/2, 1])$ is a punctured homotopy bound for constant.  
If the sizes of $f|\langle u_i,v_i\rangle$ are less than $\de$, this is a desired punctured homotopy. 
Otherwise, we extend $H$ on each $\langle u_i,v_i\rangle \times [1/2,1]$ where 
 the sizes of $f|\langle u_i,v_i\rangle$ are not less than $\de$, using the second procedure. 

The second procedure is in case $\langle a,b\rangle$ is a maximal $\Ga$-bound and 
the size of $f$ is equal to or greater than $\de$. 
There are $\{ \langle u_0,v_0\rangle , \langle u_1,v_1\rangle \} \in \Ga$ such that $u_0 <a+\de _0$ and 
$b-\de _0 <v_1$ by the assumption. 

\noindent
(Case 0): 
Suppose that there exists such a pair satisfying 
$a+\de _0\le v_0$ or $u_1\le b-\de _0$. Remark that such a pair is unique, if it exists, and $f(v_0)=f(u_1)$ and $W^{f|[v_0,u_1]}$ is closed under $\Ga$. Since 

\qquad\qquad 
$u_1-v_0\le b-a -\de _0 \le (m+1)\de _0-\de _0=m\de _0, $

\noindent
we have a punctured homotopy $G_0$ from $f|[v_0,u_1]$ to constant having holes less than $\de$. 
Since $\{ \langle u_0,v_0\rangle , \langle u_1,v_1\rangle \} \in \Ga$, we have a homotopy $G_1: [v_0,u_1]\times [1/2,1]
\to E(X,D)$ 
from $f_1: [u_0, v_1] \to E(X,D)$ to constant, where $f_1(s) = f(s)$ for $s\in [u_0, v_0] \cup [u_1, v_1]$ 
and $f_1(s) = f(v_0) = f(u_1)$ for $s\in [v_0,u_1]$. 
Let $f_2(s) = f(s)$ for $s\in [a, u_0] \cup [v_1,b]$, and $f_2(s) = f(u_0)=f(v_1)$ for $s\in [u_0, v_1]$. 
Since $f(a)=f(b)$, and the sizes of $f|[a,u_0]$ and $f|[v_1,b]$ are less than $\de /2$, 
the size of $f_2$ is less than $\de$. We define $G_2: \partial ([a,b]\times [0,1/4]) \to E(X,D)$ to be a punctured 
homotopy from $g_2$ to constant. As a whole, we define 

$H| [v_0,u_1]\times [1/2,1] = G_0$ and $H(s,t) = f(s)$ for $s\in [a, v_0]\cup [u_1, b]$, and $t\in [1/2, 1]$; 

$H| [u_0,v_1]\times [1/4,1/2] = G_1$ and $H(s,t) = f(s)$ for $s\in [a,u_0]\cup [v_1, b]$, and $t\in [1/4, 1/2]$; and 

$H|[a,b]\times [0,1/4] =G_2$. Then, $H$ is a desired punctured homotopy. 

\noindent
(Case 1) Otherwise: 
Let 
\begin{eqnarray*}
u^* &=&\inf \{ v_1\, |\, \{ \langle u_0,v_0\rangle , \langle u_1,v_1\rangle \} \in \Ga \mbox{ for some }
u_0, v_0, u_1\}; and \\
v^* &=& \sup \{ u_0\, |\, \{ \langle u_0,v_0\rangle , \langle u_1,v_1\rangle \} \in \Ga \mbox{ for some } v_0, u_1, v_1\}.
\end{eqnarray*}
Then 
$W^{f|[u^*,v^*]}$ is closed under $\Ga$. Take all maximal $\Ga$-bounds 
$\langle u_i,v_i\rangle  ( i\in I) $ such that $u^* \le u_i$ and $v_i\le v^*$. 
If $u_i < a+\de _0$, then $v_i\le b-\de _0$ and if $ b-\de _0<v_i$, then $a+\de _0\le u_i$. 
Therefore, for each $i\in I$ we have $v_i-u_i\le m\de _0$ and by the induction hypothesis 
a punctured  homotopy $H_i:[u_i, v_i]\times [1/2,1]\to E(X,D)$ whose holes are less than $\de$. 

Let $G_0[u_i,v_i]\times [1/2,1] = H_i$ for $i\in I$ and $G_0(s,t)=f(s)$ for $s\in [u^*,v^*]\sm \bigcup _{i\in I} 
\langle u_i,v_i\rangle$. 
Next, define $f_1:[u^*,v^*]\to X$ by $f_1 = f(u_i)=f(v_i)$ for $s\in [u_i, v_i]$ and 
$f_1(s)=f(s)$ for $s\in [u^*,v^*]\sm \bigcup _{i\in I} \langle a_i,b_i\rangle$. 
Since $f_1(u^*) = f_1(v^*)$, 
we have a homotopy $G_1:[u^*,v^*]\times [1/4,1/2]$ from $g_1$ to constant. 
As in Case 0, let $f_2(s) = f(s)$ for $s\in [a, u^*]\cup [v^*, b]$, and $f_2(s) = f(u^*)=f(v^*)$ for $s\in [u^*,  v^:*]$. 
Since $u^*-a\le \de _0$ and $b-v^*\le \de _0$ and $f(a)=f(b)$, the size of $f_2$ is less than $\de$. 
The hole of the punctured homotopy $G_2: \partial ([a,b]\times [0,1/4]) \to E(X,D)$ from $f_2$ to constant is less than $\de$. We define 

$H| [u^*,v^*]\times [1/2,1] = G_0$ and $H(s,t) = f(s)$ for $s\in [a, u^*]\cup [v^*, b]$, and $t\in [1/2, 1]$; 

$H| [u^*,v^*]\times [1/4,1/2] = G_1$ and $H(s,t) = f(s)$ for $s\in [a, u^*]\cup [v^*, b]$, and $t\in [1/4, 1/2]$; and 

$H|[a,b]\times [0,1/4] =G_2$. Then, $H$ is a desired punctured homotopy. 
Now, we have shown the first property for $\mathcal{F}$ to be admissible. 

To show the second property, let $\vep >0$ be given. Then, by the local simle-connectivity of 
$X$ we have $\de >0$ such that $\de < \vep /3$ and every 
loop with base point $x$ in the $\de$-neighborhood of $x$ in $X$ for $x$ in the range of 
$r\circ f_0$ is 
homotopic to the constant in the $\vep /3$ neighborhood of $x$. 
Let $f\in \mathcal{F}$ be a loop whose size is less $\de$ and $\al >0$ be given. 
We choose $\de _0>0$ as in the preceding proof of the first property for $\min (\de , \al )$ instead of 
$\de$. 
Then, tracing the preceding proof we see that the range of $H$ is in the $\vep /3$-neighborhood of 
the range of $f$. Since the size of $f$ is less than $\vep /3$, the size of $H$ is less than $\vep$. 
The other properties are satisfied as in the preceding proof. 
Now we complete the proof. 
\end{proof}
\begin{cor}\label{cor:reducedword} 
Let $X$ be a simply connected, locally simply connected metric space, 
 $x\in X$, $D$ a countable dense subset of $X$ and $f$ a proper path in $E(X,D)$. 

Then there exists a proper path $g$ in $E(X,D)$ such that $W^g$ is the reduced word of $W^f$ and 
$g$ is homotopic to $f$. 
\end{cor}
\begin{proof}
We have a proper path $f: [a,b]\to E(X,D)$ such that $W^f\equiv W$. 
We choose a maximal noncrossing inverse pairring $\Ga$ on $W^f$, which can be gotten by a straightforward induction using the axiom of choice. 
By Lemma~\ref{lmm:reducedword1} we have the reduced word $W_0$ of $W$ using $\Ga$. 
Since we need to define a path, we describe this procedure more precisely. 
We have subwords $X_i$ of $W$ such that $X_i = e$, $\ov{X_i}\cap \ov{X_j} =\emptyset$ 
for $i\neq j$, $\Ga |\ov{X_i}$ is a noncrossing inverse pairring and $\ov{W_0} = \ov{W}\sm \bigcup _i \ov{X_i}$.  
Remark that $\bigcup (\Ga |\ov{X_i})$ is a subset of $[a,b]$. We let $u_i =\inf \bigcup (\Ga |\ov{X_i})$ 
and $v_i =\sup \bigcup (\Ga |\ov{X_i})$. Then $W^{f|[u_i,v_i]} \equiv X_i$. Therefore, each 
$f|[u_i,v_i]$ is nullhomotopic by Theorem~\ref{thm:main1}. 
Define 
\[
g(x) = \left\{
\begin{array}{ll}
f(u_i) = f(v_i), & 
\text{for } x\in \langle u_i, v_i \rangle ; \\[0.1cm]
f(x),&
\text{otherwise}. 
\end{array}
\right.
\]
Then, we see $W^g\equiv W_0$ and consequently $W^f = W^g$. 
The last equation implies that $W^{fg^-} = W^f (W^g)^- = e$. 
Theorem~\ref{thm:main1} implies that $fg^-$ is homotopic to the constant map and hence 
$g$ is homotopic to $f$. 
\end{proof}

We recall that a homomorphism $h:\pi _1(X,x)\to \pi _1(Y,y)$ is a spatial homomorphism, 
if there exists a continuous map $f:X\to Y$ such that $f(x)=y$ and $f_*=h$ \cite{E:spatial}. 
\begin{thm}\label{thm:main2}
Under the same condition as that in Theorem~\ref{thm:main1}, 
for each homomorphism $h:\pi _1({\mathbb H},o) \to \pi _1(E(X,D),x)$ 
there exist a point $x^*\in X$, a path $p$ from $x^*$ to $x$ 
and a spatial homomorphism $\ov{h}:\pi _1({\mathbb H},o) \to \pi _1(E(X,D),x^*)$ 
such that $h = \vp _p \cdot\ov{h}.$  
If the image of $h$ is not finitely generated, $x^*$ and $\ov{h}$ are 
unique and $p$ is unique up to homotopy relative to the end points. 
\end{thm}
\begin{proof}
By Theorem~\ref{thm:main1} we have a homomorphism $\si _*:\pi _1(E(X,D))\to 
\pmc{$\times$}\, \, \; ^{\sigma}_{d\in D}\mathbb{Z}_d$, i.e. 
$\si _*([f]) = [W^f]$ for a loop $f: \mathbb{I}\to E(X,D)$. 
Then, by \cite[Lemma 2.9]{E:edge}, $\si ^*\circ h$ is conjugate to a standard homomorphism $\ov{h}$, i.e. 
there exists $u\in \pmc{$\times$}\, \, \; ^{\sigma}_{d\in D}\mathbb{Z}_d$ such that 
$\si ^* \circ h(x) = u^{-1}\ov{h}(x)u$ for all $x\in \pi _1({\mathbb H},o)$ and $\ov{h}(W) = V$ where 
$\ov{V} = \ov{W}$ and $V(\al ) = \ov{h}(W(\al ))$ for $\al \in \ov{W}$ (see \cite{E:free}). 

If $h(\de _n)$ is trivial for almost all $n$, then we easily have a 
a spatial homomorphism $\ov{h}:\pi _1({\mathbb H},o) \to \pi _1(E(X,D),x^*)$ 
such that $\ov{h}(\de _n) = h(\de _n)$ for every $n$. 
Then 
\cite[Lemma 2.5]{E:embed} implies that $h = \ov{h}$. In this case 
the image of $\ov{h}$ is finitely generated. 
Otherwise, i.e. 
$h(\de _n)$ is non-trivial for infinitely many $n$, then 
the above conjugator $u$ is unique \cite[Lemma 2.9]{E:edge}.  

Though $E(X,D)$ is neither a space in \cite[Section 3]{E:edge} nor one-dimensional, we easily see that a standard homomorphism 
is induced from a continuous map from $\mathbb{H}$ to $E(X,D)$. Actually, let $W_n=\ov{h}(\boldsymbol{e}_n)$, where $W_n$ is a reduced word and $\boldsymbol{e}_n$ is a homotopy class of a winding on the 
$n$-th circle, then 
$\ov{h}( \boldsymbol{e}_1\boldsymbol{e}_2\cdots \boldsymbol{e}_n\cdots ) 
=W_1W_2\cdots W_n\cdots $
and we have a loop $f$ in $E(X,D)$ such that 
$W^f = W_1W_2\cdots W_n\cdots $. Therefore, $p$ is gotten as a unique accumulation point of 
$\{ d\in D\; | \; d \mbox{ or }d^- \mbox{ appears in some }W_n\}$. 
\end{proof}

We recall a construction of a space from a group
\cite{ConnerEda:information, ConnerEda:patch}. For a given group $G$ let $X_G$ be the
topological space defined in \cite[Section 2]{ConnerEda:information}. Since this construction  
is not well-known, we explain this slightly and how the local path-connectivity in the assumption of Corollary~\ref{cor:main1}  concerns. For a group $G$, a point of the space $X_G$ is a maximal compatible family of homomorphic images of the Hawaiian earring group in $G$. Here, we admit 
the difference of finitely generated parts and are only interested in uncountable homomorphic images. The convergence of points $x_n\in X_G$ is done as follows. Choosing arbitrary $a_n\in x_n$ modulo finitely generated parts, then collection $\{ a_n\, |\, n<\om \}$ belongs to the limit point. 
By Theorem~\ref{thm:main2} we see an uncountable image of the Hawaiian earring group determines 
a point as one-dimensional case \cite[Proposition 2.1]{ConnerEda:information}. To make $\{ a_n\, |\, n<\om \}$ as a 
homomorphic image of the Hawaiian earring group we need to define a continuous map from the Hawaiian 
earring. For this purpose we can use the local path-connectivity. 
\begin{cor}\label{cor:main1}
In addition to the assumptions for a space $X$ in Theorem~\ref{thm:main1}, 
suppose that $X$ is locally path-connected. Then, 
the space $X_{\pi _1(E(X,D))}$ is homeomorphic to the original space
 $X$. 
\end{cor}
Now, we have pairwise non-isomorphic subgroups of the Hawaiian earring group $\pi _1(\mathbb{H})$.   
\begin{cor}\label{cor:subgroup}
Under the same condition as in Corollary~\ref{cor:main1}, 
$\pi _1(E(X,D))$ is isomorphic to a subgroup of $\pi _1(\mathbb{H})$, which are non-isomorphic if  spaces 
$X$ are not homeomorphic. 
\end{cor}
\begin{rmk}
The local simple-connectivity in the assumption of Theorems~\ref{thm:main1} is essential. We show the existence of a simply-connected space $X$ such that $\pi _1(E(X,D))$ contains a subgroup isomorphic to the rational group $\mathbb{Q}$ for any countable dense subset $D$ of $X$. Since any non-zero abelian subgroup of $\pmc{$\times$}\, \, \; ^{\sigma}_{d\in D}\mathbb{Z}_d$ is isomorphic to $\mathbb{Z}$, the conclusion of Theorem~\ref{thm:main1} does not hold. 
Let $C\mathbb{H}$ be the cone over 
$\mathbb{H}$ which is the quotient space of $\mathbb{H}\times [0,1]$ by collapsing $\mathbb{H}\times \{ 1\}$ to one point $*$. Obviously $C\mathbb{H}$ is simply-connected. 
Choose a countable dense subset $D$ of $C\mathbb{H}$ and also $d_n\in D$ 
so that $d_n$ converge to $(o,0)$. Apparently $C\mathbb{H}$ is simply-connected. 
Let $\it{l}_n$ be essential loops in $C\mathbb{H}\sm \{ *\}$ whose ranges converge to $(o,0)$ and 
$d_n\in D\cap \mathbb{H}\times [0,1/2]$ converge to $(o,0)$. Choose paths $p_n$ from $d_n$ to $(o,0)$ whose ranges converge to $(o,0)$. 
We construct concatenations of loops $p_nw_{d_n}p_n^-{\it l_n}p_n^-w_{d_n}^-p_n$, which 
is a typical construction in wild topology \cite{E:union, E:free}. 

Let $Seq$ be the set of finite sequences of nonnegative integers.  
For $s\in Seq,$ the length of $s$ is denoted by $lh(s)$ and $s=(s(0),\cdots ,s(n-1)),$ where $n = lh(s).$  
The sequence obtained by adding $i$ to $s$ is denoted by $s*(i).$  
Let $\Sigma = \{s\in Seq \; : \; s(i)\le i \; \text{for} \; i<lh(s) \; \text{and} \; lh(s) \ge 1\}.$  
For $s\in \Sigma,$ define numbers $a_s$ and $b_s$ inductively as follows:
\[
\left\{ 
\begin{array}{ll}
a_{(0)} = 0, \quad a_{(1)} =\displaystyle{\frac{1}{2};} &{}\\[0.2cm]
\displaystyle{b_s = a_s + \frac{1}{2^n}\cdot \frac{1}{(n+1)!}}, & \text{where} \; n=lh(s); \\[0.3cm]
 \displaystyle{a_s = b_t +\frac{1}{2^{n-1}}\cdot \frac{1}{n!}\cdot \frac{i}{n+1}}, & \text{where} \; n=lh(s), t*(i)=s.
\end{array}
\right.
\]
Define $f:[0,1] \to E(C\mathbb{H}, D)$ by:
\[
\left\{
\begin{array}{ll}
f|[a_s,b_s] \equiv p_nw_{d_n}p_n^-{\it l_n}p_n^-w_{d_n}^-p_n  & 
\text{for }s\in \Sigma \text{ with } lh(s)=n, \\[0.1cm]
f(\alpha)= (o,0) &
\text{for }\alpha \notin \bigcup _{s\in \Sigma} [a_s,b_s]. 
\end{array}
\right.
\]
Here, $f\equiv g$ for two loops $f$ and $g$ means the existence of order-homeomorphism $\vp$ between $\op{dom}(f)$ and $\op{dom}(g)$ such that $f(s) = g(\vp (s))$. 
 Since the ranges of $p_nw_{d_n}p_n^-{\it l_n}p_n^-w_{d_n}^-p_n$ converge to $(o,0),$ $f$ is a loop with basic point $(o,0).$
Define $f_n:[0,1]\to E(C\mathbb{H}, D)$ by: 
\[
\left\{
\begin{array}{ll}
f_n|[a_s,b_s] \equiv f|[a_s,b_s] & 
\text{for }s\in \Sigma \text{ with } lh(s)\ge n, \\[0.1cm]
f(\alpha)= (o,0) &
\text{for }\alpha \notin \bigcup _{s\in \Sigma , lh(s)< n} [a_s,b_s]. 
\end{array}
\right.
\]
Since each  $p_nw_{d_n}p_n^-{\it l_n}p_n^-w_{d_n}^-p_n$ is null-homotopic, $f_n$ are homotopic to $f$. 
To show that $f$ is an essential map by contradiction, we suppose that $f$ is null-homotopic, i.e. 
$H:[0,1]\times [0,1]\to E(C\mathbb{H},D)$ satisfies $H(s,1) = f(s), H(0,t) = H(1,t)=H(s,0)=(o,0)$ for $s,t\in [0,1]$. 
Consider $H^{-1}(E(C\mathbb{H},D) \sm C\mathbb{H})$, which becomes a disjoint union of $H^{-1}(C_{d_n}\sm 
\{ d_n\}$. Since $w_{d_n}$ and $w_{d_n}^-$ are essential, each interval in $[0,1]\times \{ 1\}$  corresponding to $w_{d_n}$ or $w_{d_n}^-$  is contained in a component 
of $H^{-1}(C_{d_n}\sm \{ d_n\})$ which contains another interval corresponding to $w_{d_n}^-$ or 
$w_{d_n}$ respectively. Each components of the complement of $H^{-1}(C_{d_n}\sm \{ d_n\})$ intersects 
with $H^{-1}(\{ *\})$, since every ${\it l}_m$ is essential. Therefore, every component 
of the complement of $H^{-1}(\bigcup _{d\in D}C_d\sm D)$ intersects with  $H^{-1}(\{ *\})$. 
Since there are infinitely many such components, we have an accumulation point $(u,v)$. Then, $H(u,v) 
= *$. On the other hand, since $(u,v)$ is also an accumulation point of  $H^{-1}(C_{d_n}\sm \{ d_n\})$'s, $H(u,v)$ belongs to $\mathbb{H}\times [0,1/2]$, which is a contradiction. 
Hence $f$ is essential. 

Define $\bs{0}_n\in Seq$ so that $lh(\bs{0}_n) = n$ and $\bs{0}_n (i)=0$. 
Then, we see that $f|[b_{\bs{0}_n}, a_{\bs{0}_{n-1}*(1)}]$ is homotopic to 
 $(f|[a_{\bs{0}_{n+1}}, a_{\bs{0}_n*(1)}])^{n+2}$ and consequently $f|[a_{\bs{0}_n}, a_{\bs{0}_{n-1}*(1)}]$ is homotopic to
 $(f|[a_{\bs{0}_{n+1}}, a_{\bs{0}_n*(1)}])^{n+2}$. 
Since $f|[a_s,b_s]$ depends only on the length $lh(s)$, $\pi _1(E(C\mathbb{H}, D))$ contains a subgroup 
isomorphic to $\mathbb{Q}$. 
\end{rmk}

%

%\bibliographystyle{amsplain}
%\bibliography{eda}
\providecommand{\bysame}{\leavevmode\hbox to3em{\hrulefill}\thinspace}
\providecommand{\MR}{\relax\ifhmode\unskip\space\fi MR }
\providecommand{\MRhref}[2]{%
  \href{http://www.ams.org/mathscinet-getitem?mr=#1}{#2}
}
\providecommand{\href}[2]{#2}

\end{document}